\documentclass[12pt, reqno]{amsart}
\usepackage{amsfonts,amssymb,latexsym,amsmath, amsxtra}
\usepackage{verbatim}
\usepackage[draft]{hyperref}
\usepackage{enumitem}
\usepackage{tikz}
\usepackage{cancel}
\usepackage{bm}

\newcounter{conj}
\theoremstyle{plain}
\newtheorem{theorem}{Theorem}[section]
\newtheorem*{theorem*}{Theorem}
\newtheorem{corollary}[theorem]{Corollary}
\newtheorem{conjecture}[conj]{Conjecture}
\newtheorem{lemma}[theorem]{Lemma}
\newtheorem{proposition}[theorem]{Proposition}
\newtheorem*{conjecture*}{Conjecture}

\theoremstyle{definition}

\theoremstyle{remark}

\newtheorem*{remark*}{Remark}
\newtheorem*{remarks*}{Remarks}

\numberwithin{equation}{section}

\newcommand{\R}{\mathbb R}
\newcommand{\N}{\mathbb N}

\def\({\left(}
\def\){\right)}

\begin{document}

\title{Fractional partitions and conjectures of Chern--Fu--Tang and Heim--Neuhauser}
\author{Kathrin Bringmann}
\address{University of Cologne, Department of Mathematics and Computer Science, Weyertal 86-90, 50931 Cologne, Germany}
\email{kbringma@math.uni-koeln.de}
\author{Ben Kane}
\address{Department of Mathematics, University of Hong Kong, Pokfulam, Hong Kong}
\email{bkane@hku.hk}
\author{Larry Rolen}
\author{Zack Tripp}
\address{Department of Mathematics, Vanderbilt University, Nashville, TN 37240}
\email{larry.rolen@vanderbilt.edu}
\email{zachary.d.tripp@vanderbilt.edu}

%\nocite{*}

\begin{abstract}
Many papers have studied inequalities for partition functions. Recently, a number of papers have considered mixtures between additive and multiplicative behavior in such inequalities. In particular, Chern--Fu--Tang and Heim--Neuhauser gave conjectures on inequalities for coefficients of  powers of the generating partition function. These conjectures were posed in the context of colored partitions and the Nekrasov--Okounkov formula. Here, we study the precise size of differences of products of two such coefficients. 
This allows us to prove the Chern--Fu--Tang conjecture  and to show the Heim--Neuhauser conjecture in a certain range. The explicit error terms provided will also be useful in the future study of partition inequalities. These are laid out in a user-friendly way for the researcher in combinatorics interested in such analytic questions.
\end{abstract}
\maketitle

\section{Introduction and Statement of Results}\label{intro}
The estimation of partition functions has a long history. Hardy and Ramanujan \cite{HR} initiated this subject by proving the asymptotic formula
\[
p(n)\sim\frac{1}{4\sqrt{3}n}e^{\pi\sqrt{\frac{2n}{3}}}\qquad (n\to\infty)
\]
for the integer partition function $p(n)$. The proof relies on the modularity properties of the \emph{Dedekind-eta function}, $\eta(\tau):=q^{\frac{1}{24}}\prod_{n\geq1}(1-q^n)$ with $q:=e^{2\pi i \tau}$. The partition function is connected to the $\eta$-function by the generating function formula:
\[
\sum_{n\geq0}p(n)q^n=\frac{q^{\frac{1}{24}}}{\eta(\tau)}.
\]
Hardy and Ramanujan's proof birthed the {Circle Method}, which is now an important tool in  analytic number theory (see, e.g. \cite{V}); 
Hardy and Ramanujan also proposed a divergent series for $p(n)$, which Rademacher \cite{R} improved to give an exact formula for $p(n)$. We now know that this was an early example of a Poincar\'e series, and this has been generalized in many directions \cite{BFOR}. 

The analytic properties of related functions have frequently been studied. For instance, many people investigated the $\alpha$-th power $\eta(\tau)^\alpha$ of the Dedekind $\eta$-function. For $\alpha = 24$, one has the famous {modular discriminant} $\Delta(\tau)$. Ramanujan's original conjecture on the growth of the coefficients of $\Delta(\tau)$ has been hugely influential in the general theory of $L$-functions and automorphic forms \cite{Sar}. It also remained unsolved until it was shown as a consequence of Deligne's proof of the Weil conjectures \cite{D}.
More generally, positive powers have been studied in seminal works of Dyson~\cite{Dy} and Macdonald~\cite{M}, and encode important Lie-theoretic data thanks to the Macdonald identities \cite{M}.
For negative integral powers, one obtains {colored partition} generating functions. Specifically, for $k\in \N$, 
\[\frac{1}{\eta(\tau)^{k}}=:q^{-\frac{k}{24}}\sum_{n\geq0}p_{k}(n)q^n\]
is the generating function for the number of ways to write the number $n$ as a sum of positive integers using $k$ colors. We consider the coefficients of $\eta(\tau)^{-\alpha}$ for  arbitrary positive real $\alpha$, although the coefficients no longer have the same combinatorial meaning in counting colored partitions. However, the insertion of a continuous parameter $\alpha$ still gives important information. The most important instance of this is thanks to the famous \emph{Nekrasov--Okounkov formula} \cite{NO}
\begin{equation}\label{NOFormula}
\sum_{\lambda\in\mathcal P}q^{|\lambda|}\prod_{h\in\mathcal H(\lambda)}\left(1-\frac{\alpha}{h^2}\right)=q^{\frac{1-\alpha}{24}}\eta(\tau)^{\alpha-1}.
\end{equation}
Here, $\mathcal P$ is the set of all integer partitions, $|\lambda|$ denotes the number being partitioned by $\lambda$, and $\mathcal H(\lambda)$ is the multiset of hook lengths of $\lambda$. 
This formula arose from their study of supersymmetric gauge theory and a corresponding statistical-mechanical partition function, and is related to random partitions.

In several recent papers, Heim, Neuhauser, and others \cite{HN3, HN4, HN2} have studied the analytic properties of the Nekrasov--Okounkov formulas. For a fixed $n$, the $n$-th Fourier coefficient of \eqref{NOFormula} is a polynomial in $\alpha$, which Heim and Neuhauser conjectured to be unimodal. Partial progress towards this result was recently given by Hong and Zhang \cite{HZ}. On the other hand, considering all of the coefficients of \eqref{NOFormula} for a fixed $\alpha$ led Heim and Neuhauser to make their conjecture below. In order to explain the context of their conjecture further, we now discuss a related chain of partition inequalities which has recently received attention. 
Independent work by Nicolas \cite{N} and DeSalvo and Pak \cite{DSP} proved that the partition function $p(n)$ is eventually log-concave, specifically, that 
\[
p(n)^2-p(n-1)p(n+1)\geq0
\]
for all $n > 25$. 
This result was vastly generalized to a conjecture for certain higher degree polynomials, arising from so-called {Jensen polynomials} by Chen, Jia, and Wang \cite{CJW}. That generalized version was later proven by Griffin, Ono, Zagier, and the third author \cite{GORZ}. 

Expanding in another direction, Bessenrodt and Ono \cite{BO} showed that the partition function satisfies mixed additive and multiplicative properties. Specifically, they showed that for all integers $a,b\geq2$ with $a+b>8$, one has 
\[
p(a)p(b)\geq p(a+b). 
\]

Extensions of this result, both rigorous and conjectural, have since been proposed by a number of authors. Alanazi, Gagola, and Munagi \cite{AGM} gave a combinatorial proof of this result, while Heim and Neuhauser studied the inequality given by replacing the argument $a+b$ by $a+b+m-1$ \cite{HN}. Similar inequalities that mix additive and multiplicative properties for different types of partition statistics have been studied as well \cite{BB, DM, HJ}.

The first conjecture which we study was made by Chern, Fu, and Tang, who proposed the following analogous conjecture for colored partitions. 
\begin{conjecture}[Chern, Fu, and Tang, Conjecture 5.3 of \cite{CFT}]\label{Conj CFT}
For $n,\ell\in\mathbb{N}$, $k\in\mathbb{N}_{\geq2}$, if $n>\ell$ and $(k,n,\ell)\neq(2,6,4)$, we have
\[
p_{k}(n-1)p_{k}(\ell+1)\geq p_{k}(n)p_{k}(\ell).
\]
\end{conjecture} 
\begin{remark*}
As noted in a paper by Sagan \cite{S}, Conjecture \ref{Conj CFT} is equivalent for $k \ge 3$ to the log-concavity of $p_k(n)$. 
\end{remark*}

Heim and Neuhauser conjectured a continuous extension.
\begin{conjecture}[Heim and Neuhauser, \cite{H}]\label{Conj HN}
Under the same assumptions, Conjecture \ref{Conj CFT} still holds if $k$ is replaced by $\alpha\in\R_{\ge 2}$. 
\end{conjecture}
\begin{remark*}
As stated, the conjecture is not quite true; by writing the polynomials $p_\alpha(4)$, $p_\alpha(5)$, and $p_\alpha(6)$ and considering the inequality $p_\alpha(5)^2 - p_\alpha(4)p_\alpha(6)\ge 0$, we see that additional exceptions are needed above. Namely, if we define $\alpha_0 \approx 2.055$ to be the largest real root of the irreducible polynomial $z^7 + 42 z^6 + 684 z^5 + 4038 z^4 + 13119 z^3 + 12048 z^2 - 100204 z - 59328$, then the exemption of $(\alpha, n, \ell) \neq (2,6,4)$ in the conjecture should be changed to $(\alpha, n, \ell) \notin \{ (\alpha, 6, 4): 2 \le \alpha < \alpha_0 \}$. 
\end{remark*}

We study these conjectures with the aim of writing down explicit results which may be of use for the future of related inequalities. To do this, we consider the sign of the general difference of products:
\[
p_{\alpha_1}(n_1)p_{\alpha_2}(n_2)-p_{\alpha_3}(n_3)p_{\alpha_4}(n_4),
\]
for any $n_1,n_2,n_3,n_4\in\N$ and $\alpha_1,\alpha_2,\alpha_3,\alpha_4\in\R^+$.  This study leads to our first main result.
\begin{theorem}\label{trivial theorem}
Fix $\alpha_1, \alpha_2, \alpha_3, \alpha_4 \in \mathbb{R}^+$, and consider the inequality 
\begin{equation*}
p_{\alpha_1}(n_1)p_{\alpha_2}(n_2) \ge p_{\alpha_3}(n_3)p_{\alpha_4}(n_4).
\end{equation*}
Without loss of generality, we assume $n_1 \ge n_2$ and $n_3 \ge n_4$. 
If $n_3 = o(n_1)$, the inequality is true for $n_1$ sufficiently large. Conversely, if $n_1 = o(n_3)$, the inequality is false for $n_3$ sufficiently large. 
\end{theorem}
Theorem~\ref{trivial theorem} can be made explicit. This is applied below to prove the  conjectures of Chern--Fu--Tang and Heim--Neuhauser. Here and throughout the paper, we use the notation $f(x) = O_\le(g(x))$ to mean $|f(x)| \le g(x)$ for a positive function $g$ and for all $x$ in the domain in which the functions are defined.
\begin{theorem}\label{main theorem}
Fix $\alpha \in \R_{\ge 2}$, and let $n, \ell \in \N_{\geq2}$ with $n > \ell + 1$. Set $N:= n-1 - \frac{\alpha}{24}$ and $L := \ell - \frac{\alpha}{24}$, we suppose $L \ge \max\{2\alpha^{11}, \frac{100}{\alpha - 24}\}$. Then we have
\begin{multline*}
p_\alpha(n-1)p_\alpha(\ell+1) - p_\alpha(n)p_\alpha(\ell) %\nonumber 
\\
 = \pi \left(\frac{\alpha}{24}\right)^{\frac{\alpha}{2}+1} N^{-\frac{\alpha}{4}-\frac 54}L^{-\frac{\alpha}{4}-\frac 54}e^{\pi\sqrt{\frac{2\alpha}{3}}\left(\sqrt{N}+\sqrt{L}\right)} \left(\sqrt{N} - \sqrt{L}\right) \left(1 + O_\le\left(\frac{14}{15}\right)\right).
\end{multline*}
\end{theorem}
Because the last expression in parentheses in Theorem \ref{main theorem} is always positive, Conjecture \ref{Conj HN} is true for $\ell$ sufficiently large. Note that Conjecture \ref{Conj HN} is trivially true if $n = \ell + 1$, which is why the theorem is sufficient.

\begin{corollary}\label{eventual}
Conjecture \ref{Conj HN} is true for $\ell \ge \max\{ 2\alpha^{11} + \frac{\alpha}{24}, \frac{100}{\alpha - 24} + \frac{\alpha}{24}\}$. 
\end{corollary}

Additionally, for some $\alpha\in\N$ we are able to numerically verify that the inequality still holds for small values of $\ell$ and $n$, giving the following corollary.

\begin{corollary}\label{log-concavity}
	Conjecture \ref{Conj CFT} is true. In particular, $p_2(n)$ is log-concave for $n \ge 6$, and $p_{k}(n)$ is log-concave for all $n$ and $k \in \N_{\ge 3}$. 
\end{corollary} 

\begin{remark*} 
	Although Theorem~\ref{main theorem} turns Conjecture~\ref{Conj CFT} into a finite computer check, the number of cases that must be checked to give Corollary~\ref{log-concavity} is very large. Thus, direct brute force computer checks are not sufficient. Faster methods of verifying such inequalities are described in the proofs below. These may be useful in future partition investigations.
\end{remark*}
The remainder of the paper is organized as follows. We review basic ingredients needed for the proofs of our theorems in Section~\ref{Prelim}. These proofs are then carried out in Section~\ref{Proofs}. In Section~\ref{corollary}, we provide lemmas and discussion needed for our computations in order to prove our corollary. We then conclude in Section~\ref{conclusion} with some ideas for further work.

\section*{Acknowledgements}
The authors thank Ken Ono for proposing this project and Bernhard Heim for useful comments on an earlier draft. Moreover, we thank the referee for helpful comments.
The research of the first author is supported by the Alfried Krupp Prize for Young University Teachers of the Krupp foundation. The research of the second author was supported by grants from the Research Grants Council of the Hong Kong SAR, China (project numbers HKU 17301317 and 17303618). This project has received funding from the European Research Council (ERC) under the European Union's Horizon 2020 research and innovation programme (grant agreement No. 101001179).

\section{Preliminaries}\label{Prelim}
Here, we review the key ingredients for the proof of our results.
\subsection{Exact formulas for partitions}\label{Exact}
In a recently submitted paper, Iskander, Jain, and Talvola \cite{IJT} gave an exact formula for the fractional partition function in terms of Kloosterman sums and Bessel functions. The \textit{$\alpha$-Kloosterman sum} is given by
\begin{equation*}
A_{k,\alpha}(n,m):= \sum\limits_{\substack{0 \le h < k\\ \text{gcd}(h,k)=1}} e^{\pi i \alpha s(h,k) + \frac{2\pi i}{k}(m\bar{h} - n)h},
\end{equation*}
where $\bar{h}$ denotes the inverse of $h$ modulo $k$ and $s(h,k)$ is the usual Dedekind sum. 
The only properties we need of this sum are that $A_{1,\alpha}(n,m) = 1$ and $|A_{k,\alpha}(n,m)| \le k$. We have the following result from \cite{IJT}. 
\begin{theorem} \label{theorem} 
	For all $\alpha \in \mathbb{R}^+$ and $n > \frac{\alpha}{24}$, we have
	\begin{multline*}
	p_\alpha(n) =  2\pi \left(n - \frac{\alpha}{24}\right)^{-\frac{\alpha}{4}-\frac 12} \sum\limits_{m=0}^{\lfloor \frac{\alpha}{24}\rfloor} \left( \frac{\alpha}{24}-m\right)^{\frac{\alpha}{4}+\frac{1}{2}} p_\alpha(m) \\\times \sum\limits_{k\geq 1} \frac{A_{k,\alpha}(n,m)}{k} I_{\frac \alpha2 + 1} \left(\frac{4\pi}{k} \sqrt{\left(\frac{\alpha}{24}-m\right)\left(n - \frac{\alpha}{24}\right)}\right).
	\end{multline*}
\end{theorem}
This provides an exact formula for the numbers we wish to estimate. The difficulty lies in providing precise estimates for the error terms after truncating the series to a finite number of terms in the sum on $k$. The analysis required for these estimates is continued in the next subsection.

\subsection{Explicit bounds for Bessel functions}

In order to make the exact formula in Theorem~\ref{theorem} useful for our purposes, we need strong estimates on the Bessel functions. Although many Bessel function estimates are standard and a whole asymptotic expansion is known \cite[equation~10.40.1]{DLMF}, we were unable to find existing bounds suitable for our purposes. Thus, we describe some basic estimates here and sketch our proofs for them. In particular, we prove the following.

\begin{lemma}\label{Bessel}
	Let $\kappa \in \R$ with $\kappa > - \frac 12$.
	\begin{enumerate}
		[wide, labelwidth=!, labelindent=0pt] 
		\item [\normalfont(1)] For $x\ge 1$, we have 
		\begin{equation*}
		I_\kappa(x) \le \sqrt{\frac{2}{\pi x}}e^x.
		\end{equation*}
		\item [\normalfont(2)] For $x \ge \frac{a^6}{120}$ and $a \ge \frac 52$, we have
		\begin{equation*}
			\Gamma(a,x)\le \frac{52}{17} x^{a-1}e^{-x}.
		\end{equation*}
		\item [\normalfont(3)] For $0 \le x < 1$, we have
		
		\begin{equation*}
		I_\kappa(x) \le \frac{2^{1-\kappa}x^{\kappa}}{\Gamma(\kappa + 1)}.
		\end{equation*}
		\item [\normalfont(4)] For $\kappa\ge 2$ and $x \ge \frac{1}{120}(\kappa + \frac 72)^6$, we have {\small
		\begin{multline*}
			\left|I_\kappa(x)e^{-x}\sqrt{2\pi x}-1+\frac{4\kappa^2 - 1}{8x} - \frac{(4\kappa^{2} -1)(4\kappa^{2} -9)}{128x^2} 
			+\frac{\left(4\kappa^2-1\right)\left(4\kappa^2-9\right)\left(4\kappa^2-25\right)}{3072x^3}
			\right| \le \frac{31\kappa^8}{6x^4}.
		\end{multline*} }
	\end{enumerate}
\end{lemma}

\begin{remark*}
	We note that similar estimates needed for Lemma~\ref{Bessel} (1) also appear in Section 4.1 of \cite{IJ}. For the reader's convenience, we provide a proof here.
\end{remark*}

\begin{proof}[Proof of Lemma~\ref{Bessel}]
	(1) We use the following integral representation (see page 172 of \cite{W}):
	\begin{equation}\label{integral}
	I_\kappa(x)=\frac{\left(\frac{x}{2}\right)^{\kappa}}{\Gamma\left(\kappa+\frac 12\right)\sqrt{\pi}}\int_{-1}^{1}\left(1-t^2\right)^{\kappa-\frac 12}e^{xt}dt.
	\end{equation}
	In \eqref{integral}, naively bound the integral from $-1$ to $0$ against the integral from $0$ to $1$ giving an extra factor of $2$. Making the change of variables $u=1-t$, the remaining integral equals
	\begin{equation}\label{changes}
	e^x\int_{0}^{1}(2-u)^{\kappa-\frac 12}u^{\kappa-\frac 12}e^{-xu}du
	\leq
	2^{\kappa-\frac 12}e^xx^{-\kappa-\frac 12}\Gamma\left(\kappa+\frac 12\right).
	\end{equation}
	Plugging into \eqref{integral} gives the claim. \newline 
	(2) We begin with an upper bound coming from \cite[Theorem 1.1]{P}, namely
	\begin{equation}\label{incomplete gamma upper bound}
	 	\Gamma(a,x) < \frac{(x+b_a)^a - x^a}{ab_a}e^{-x}
	 \end{equation}
	for $a > 2$, where $b_a:= \Gamma(a+1)^{\frac{1}{a-1}}$. We wish to bound the right-hand side of \eqref{incomplete gamma upper bound} by an explicit constant times $x^{a-1}e^{-x}$. To do so, we first apply Taylor's Theorem to the function $y\mapsto (x+y)^a$. This yields
	\begin{equation*}
		(x+b_a)^a = x^a + ab_a(x+\xi)^{a-1}
	\end{equation*}
	for some $\xi \in [0, b_a]$. Thus, 
	\begin{equation*}
		\frac{(x+b_a)^a - x^a}{ab_a} = (x+\xi)^{a-1} \le (x+b_a)^{a-1}.
	\end{equation*}
	From \eqref{incomplete gamma upper bound}, we may then write
	\begin{equation*}
		\Gamma(a,x) < (x+b_a)^{a-1}e^{-x} = \left( 1 + \frac{b_a}{x} \right)^{a-1} x^{a-1}e^{-x}.
	\end{equation*}
	To complete our proof, we only need to bound the quantity $(1+ \frac{b_a}{x})^{a-1}$ by a constant. By assumption, $x > \frac{a^6}{120}$, while using \cite[equation 5.6.1]{DLMF} and basic calculus, one may find that $b_a < \frac{9a}{10}$ for $a \ge \frac 52$. Combining these bounds, we find that 
	\begin{equation*}
		\left( 1 + \frac{b_a}{x} \right)^{a-1} \le \left( 1 + \frac{108}{a^5} \right)^{a-1} < \frac{52}{17},
	\end{equation*} 
	where the last inequality follows by standard optimization techniques for $a \ge \frac 52$.\newline
	(3) This follows directly  from equation (6.25) of \cite{L}.\newline 
	(4)
	We consider first the integral from $0$ to $1$ which is on the left-hand side of (\ref{changes}).
	Now write, using Taylor's Theorem,
	\begin{multline}\label{Taylor}
		(2-u)^{\kappa-\frac 12}=2^{\kappa-\frac 12}-\left(\kappa -\frac 12\right)2^{\kappa-\frac 32}u + \frac{1}{2}\left(\kappa-\frac{1}{2}\right)\left(\kappa-\frac{3}{2}\right)2^{\kappa-\frac{5}{2}}u^2 
		\\-\frac{1}{6}\left(\kappa-\frac 12\right)\left(\kappa-\frac 32\right)\left(\kappa-\frac 52\right)2^{\kappa-\frac 72}u^3
		+C_\kappa(u) u^4,
	\end{multline}
	where for some $\xi\in[0,1]$, 
	\begin{equation*}
		C_\kappa(u):=\frac {1}{4!}\left[\frac{\partial^4}{\partial u^4}(2-u)^{\kappa-\frac 12}\right]_{u=\xi}
		=\frac{1}{4!} \left(\kappa-\frac 12\right)\left(\kappa-\frac 32\right)\left(\kappa -\frac{5}{2}\right)\left(\kappa-\frac 72\right)\left(2-\xi\right)^{\kappa-\frac 92}.
	\end{equation*}
	We can bound this by 
	\begin{equation*}
		|C_{\kappa}(u)| \leq  \frac{1}{4!} \left(\kappa - \frac{1}{2}\right)\left(\kappa-\frac{3}{2} \right)\left|\kappa-\frac{5}{2}   \right|\left|\kappa-\frac 72\right|\operatorname{max}\left\{2^{\kappa -\frac{9}{2}},1\right\} .
	\end{equation*}
	\hspace*{1,8mm}
	We first consider the contribution from the first 4 terms in \eqref{Taylor}. These are
	\begin{multline}\label{mainterm}
		2^{\kappa-\frac 12}\left(\int_{0}^{\infty}-\int_{1}^{\infty}\right)\left(1-\frac{1}{2}\left(\kappa - \frac 12\right)u + \frac{1}{8} \left(\kappa-\frac{1}{2}\right)\left(\kappa-\frac{3}{2}\right)u^2
		\right.\\\left.
		-\frac{1}{48}\left(\kappa-\frac 12\right)\left(\kappa-\frac 32\right)\left(\kappa-\frac 52\right)u^3
		\right)u^{\kappa+\frac 12}e^{-xu}\frac{du}{u}.
	\end{multline}
	Evaluating the first integral yields the main term. \\
	\hspace*{3,8mm} The second integral in \eqref{mainterm} contributes 
	\begin{multline*}
		-2^{\kappa-\frac 12}x^{-\kappa-\frac 12}\left(\Gamma\left(\kappa+\frac 12,x\right)-\frac{\kappa - \frac 12}{2x}\Gamma\left(\kappa+\frac 32,x\right)
		+ \frac{1}{8x^2}\left(\kappa-\frac{1}{2}\right)\left(\kappa-\frac{3}{2}\right)\Gamma\left(\kappa+\frac{5}{2},x\right)
		\right.\\\left.
		-\frac{1}{48x^3}\left(\kappa-\frac 12\right)\left(\kappa-\frac 32\right)\left(\kappa-\frac 52\right)\Gamma\left(\kappa+\frac 72,x\right)
		\right).
	\end{multline*}
	Using part (2), one can show that this term overall contributes at most
	\begin{equation*}
		\frac{\frac{52}{17}\left(\frac{\kappa^3}{48}+\frac{\kappa^2}{32}+\frac{71\kappa}{192}+\frac{103}{128}\right)}{\Gamma\left(\kappa + \frac{1}{2}\right)\sqrt{2\pi}}x^{\kappa-1}.
	\end{equation*}
	\hspace*{1,8mm}
	We next estimate the term with $C_\kappa(u)$ in \eqref{Taylor}. Bounding the integral from $0$ to $1$ against the integral from $0$ to $\infty$, this term can be bounded against 
	\begin{equation*}
		\frac{e^x}{\sqrt{2\pi x}}\frac{\left(\kappa^2-\frac{1}{4}\right)\left(\kappa^2-\frac{9}{4}\right)\left|\kappa^{2} - \frac{25}{4}\right|\left|\kappa^{2} - \frac{49}{4}\right|}{12x^4}  2^{-\kappa-\frac{1}{2}}\operatorname{max}\left\{2^{\kappa-\frac{9}{2}},1\right\}.
	\end{equation*}
	\hspace{3,8mm}Finally, the contribution from the integral from $-1$ to $0$ can be bounded by (estimating the integrand trivially)
	\begin{equation*} \frac{\left(\frac{x}{2}\right)^\kappa}{\Gamma\left(\kappa+\frac 12\right)\sqrt{\pi}}.
	\end{equation*}
	Overall we obtain
	\begin{align*}
		&\left|I_\kappa(x)e^{-x}\sqrt{2\pi x}-1-\frac{1-4\kappa^2}{8x}-\frac{(4\kappa^2-1)(4\kappa^2-9)}{128x^2}
		-\frac{\left(1-4\kappa^2\right)\left(9-4\kappa^2\right)\left(25-4\kappa^2\right)}{384x^3}
		\right| \\
		\leq&
		\left(\frac{3\kappa^3x^{\kappa+\frac 92}e^{-x}}{4\Gamma\left(\kappa+\frac 12\right)}+\frac{1}{12}\left(\kappa^2-\frac 14\right)\left(\kappa^2-\frac 94\right) \left|\kappa^2-\frac{25}{4}\right|\left|\kappa^2-\frac{49}{4}\right|
		2^{-\kappa-\frac{1}{2}}\operatorname{max}\left\{2^{\kappa-\frac{9}{2}},1\right\}\right) \frac{1}{x^4}.
	\end{align*}
	By elementary bounds \cite[equation~5.6.1]{DLMF}, we find
	\begin{equation*}
		\frac{3\kappa^3x^{\kappa + \frac 92} e^{-x}}{4\Gamma\left(\kappa+ \frac 12\right)} \le \frac{3\kappa^3}{4} \sqrt{\frac{\kappa+\frac 92}{2\pi}}\left(\kappa+\frac{1}{2}\right)\left(\kappa+\frac 32\right)\left(\kappa +\frac{5}{2}\right)\left(\kappa +\frac{7}{2}\right).
	\end{equation*}
	Combining the above now easily gives the claim.
\end{proof}

\section{Proofs of the theorems}\label{Proofs}
\begin{proof}[Proof of Theorem \ref{trivial theorem}]
\maketitle
We use the exact formula from Theorem \ref{theorem} and note that the dominant term comes from $m=0$ and $k =1$. The claim then follows from $I_{\kappa}(x)$ $\sim$ $\frac{e^x}{\sqrt{2\pi x}}$ as $x \rightarrow \infty.$
\end{proof}
\begin{proof}[Proof of Theorem \ref{main theorem}]
%We are now in a position to prove Theorem~\ref{main theorem}. 
Note that the claim is trivially true when $N = L$, so we assume that $N > L$ throughout. We again use the exact formula from Theorem \ref{theorem} and note that the dominant term comes from $m = 0$ and $k = 1$ in each expansion. We see that this main term in $p_\alpha(n-1)p_\alpha(\ell+1)-p_\alpha(n)p_\alpha(\ell)$ is
\begin{multline}\label{main}
4\pi^2\left(\frac{\alpha}{24}\right)^{\frac{\alpha}{2}+1}
\frac{I_{\frac \alpha 2 +1}\left( \pi\sqrt{\frac {2\alpha}{3} N}\right) }{N^{\frac{\alpha}{4}+\frac 12}}
\frac{I_{\frac \alpha 2 +1}\left( \pi\sqrt{\frac{ 2\alpha}{3} (L+1)}\right) }{(L+1)^{\frac{\alpha}{4}+\frac 12}}
\\
-4\pi^2\left(\frac{\alpha}{24}\right)^{\frac{\alpha}{2}+1}
\frac{I_{\frac \alpha 2 +1}\left( \pi\sqrt{\frac{2\alpha}{3} (N+1)}\right) }{(N+1)^{\frac{\alpha}{4}+\frac 12}}
\frac{I_{\frac \alpha 2 +1}\left( \pi\sqrt{\frac{2\alpha}{3} L}\right)}{L^{\frac{\alpha}{4}+\frac 12}},
\end{multline}
where $N:=n-1-\frac{\alpha}{24}$, $L:=\ell-\frac{\alpha}{24}$. To rewrite the Bessel functions as sums of powers of $N$ and $L$, we note that
\begin{equation*}
	\pi \sqrt{\frac{2\alpha N}{3}} \ge \pi \sqrt{\frac{2\alpha\cdot 2\alpha^{11}}{3}} \ge \frac{1}{120}\left(\frac{\alpha}{2} + \frac 92\right)^6,
\end{equation*}
where the last inequality may be checked using calculus. Hence, we are able to apply Lemma~\ref{Bessel} (4) with $\kappa = \frac{\alpha}{2} + 1$ to obtain 
\begin{equation}\label{I expansion}
I_{\frac{\alpha}{2}+1}\left(\pi \sqrt{\frac{2\alpha N}{3}} \right) = \frac{3^{\frac 14}e^{\pi \sqrt{\frac{2\alpha N}{3}}}}{2^{\frac 34} \pi \alpha^{\frac 14} N^{\frac 14}} \left(1 + \frac{c_{\alpha,1}}{N^{\frac 12}} + \frac{c_{\alpha,2}}{N} + \frac{c_{\alpha,3}}{N^{\frac 32}} + \frac{D_{\alpha, 1}(N)}{{N^{2}}}\right),
\end{equation}
where
\begin{align}c_{\alpha,1} =  O_\le \left( \frac{\alpha ^{\frac{3}{2}}}{5}\right), ~ 
c_{\alpha,2} = O_\le \left( \frac{\alpha^3}{128}\right),~   
c_{\alpha,3} = O_\le \left( \frac{\alpha^{\frac 92}}{3500} \right), ~
D_{\alpha, 1}(N) = O_\le \left( \frac{3\alpha^6}{25} \right).  \label{ca}
\end{align}
Thus we have 
\begin{multline}\label{prod of Bessels}
\frac{I_{\frac \alpha 2 +1}\left( \pi\sqrt{\frac {2\alpha (N+1)}{3}}\right)}{(N+1)^{\frac{\alpha}{4}+\frac 12}}
\frac{ I_{\frac \alpha 2 +1}\left( \pi\sqrt{\frac {2\alpha L}{3}}\right)}{L^{\frac{\alpha}{4}+\frac 12}} \\
= 
\frac{\sqrt{3}e^{\pi\sqrt{\frac{2\alpha}{3}}\left(\sqrt{N+1} + \sqrt L \right)}}{2\sqrt{2\alpha}\pi^2 (N+1)^{\frac \alpha4+\frac 34} L^{\frac\alpha4+\frac 34}}
\left( 1+\frac{c_{\alpha,1}}{(N+1)^{\frac 12}} + \frac{c_{\alpha, 2}}{N+1} + \frac{c_{\alpha,3}}{(N+1)^{\frac 32}} + \frac {D_{\alpha, 1}(N+1)}{(N+1)^{2}} \right) \\
 \times
\left( 1+\frac{c_{\alpha,1}}{L^\frac 12} + \frac{c_{\alpha,2}}{L} + \frac{c_{\alpha, 3}}{L^{\frac 32}} +\frac {D_{\alpha, 1}(L)}{L^{2}}\right) .
\end{multline}
As alluded to above, we wish to expand the main term into sums of powers of $N$'s and $L$'s, so we change the $(N+1)$'s above into $N$'s. First, note that by using Taylor's Theorem, there exist $D_{A,2}^{**}$, $D_{A,2}^*$, and $D_{A,2}$ such that
\begin{equation}\label{NA}
(N+1)^{-A}=D_{A,2}^{**}N^{-A}=N^{-A}\left( 1+\frac{D_{A,2}^*}{N}\right) = N^{-A}\left( 1 - \frac{A}{N} + \frac{D_{A,2}}{N^2}\right).
\end{equation}
Explicitly bounding in the interval $[0,1]$, Taylor's Theorem further tells us that for $A > 0$,
\begin{equation*}
|D_{A,2}^{**}|\leq 1, \quad |D_{A,2}^{*}|\leq A, \quad |D_{A,2}| \le \frac{A(A+1)}{2}.
\end{equation*}
Using this, one can prove the bounds
\begin{align}
\left|D_{\frac\alpha4+\frac 74,2}^*\right|&\leq \frac{9\alpha}{8}, \quad \left|D_{\frac\alpha4+\frac 94,2}^*\right|\leq \frac{11\alpha}{8}, \quad
\left|D_{\frac \alpha4+ \frac 34,2}\right|\leq \frac{45\alpha^2}{128}, \quad
\left|D_{\frac \alpha4+ \frac 54,2}\right| \le \frac{77\alpha^2}{128}.\label{Da2}
\end{align}
In addition to rewriting the powers of $N+1$ in \eqref{prod of Bessels} as powers of $N$, we also want to replace the $\sqrt{N+1}$ in $e^{\pi \sqrt{\frac{2\alpha}{3}(N+1)}}$ by a function of $\sqrt{N}$ instead. This is needed in order to compare the two summands of our main term. To do so, we show that for some $D_{\alpha,3}(N)\in\R$,
\begin{align}
\frac{e^{\pi\sqrt{\frac{2\alpha}{3}(N+1)}}}{e^{\pi\sqrt{\frac{2\alpha}{3}N}}} = e^{\pi\sqrt{\frac{2\alpha}{3} N}\left(\sqrt{1+\frac{1}{N}}-1\right)}=1+\frac{\pi\alpha^\frac12}{\sqrt{6}N^{\frac 12}}+\frac{\pi^2\alpha}{12N}+ \frac{\pi^3 \alpha^{\frac 32} - 9\pi\alpha^\frac 12}{36\sqrt{6}N^{\frac 32}}+\frac{D_{\alpha,3}(N)}{N^{2}}.\label{Nex}
\end{align}
To prove the second equality and determine a bound for $D_{\alpha,3}(N)$, we write $G(x):=e^{cg(x)}$ with $g(x):=\frac{\sqrt{1+x^2}}{x}-\frac{1}{x}$ and $c:=\pi \sqrt{\frac{2\alpha}{3}}$. The middle term of \eqref{Nex} is equal to $G(\frac1{\sqrt{N}})$,  so the equality is proved just by taking the first four terms of the Taylor expansion of $G(x)$ about $x = 0$ and plugging in $x = \frac{1}{\sqrt{N}}$. To bound $D_{\alpha, 3}(N)$, note that by Taylor's Theorem it is equal to $\frac{G^{(4)}(\xi)}{4!}$ for some $\xi \in [0, \frac1{\sqrt{N}}]$, so we need to bound $G^{(4)}(\xi)$ on this interval. Using some basic calculus, one finds that for $x \in [0,1]$ 
\begin{equation*}
\left|g'(x)\right|\leq\frac 12, \quad \left|g''(x)\right|<\frac{3}{10}, \quad \left|g^{(3)}(x)\right|\le \frac 34, \quad \left|g^{(4)}(x)\right|\le \frac 85.
\end{equation*}
 Moreover, $g'(x) > 0$ on $[0,1]$, so we have that 
\begin{equation*}
g(\xi)\leq g\left(\frac{1}{\sqrt{N}}\right)=\sqrt{N+1}-\sqrt{N} \le \frac{1}{2\sqrt{N}}.
\end{equation*}
Combining these estimates on $g(x)$ and its derivatives, one sees that 
\begin{equation*}
|D_{\alpha,3}(N)|\leq \frac{\left|G^{(4)}(\xi)\right|}{24}\le \frac{31 \alpha^{2}}{48}  e^{\pi\sqrt{\frac{\alpha}{6N}}}.
\end{equation*}
Assuming that $N \ge 2\alpha^{11}$, we obtain in this region
\begin{equation}\label{Da3}
|D_{\alpha, 3}(N)| \le \frac{31 \alpha^2}{48}  e^{\frac{\pi}{2\sqrt{3}\alpha^{5}}} \le \frac{2 \alpha^2}{3}.
\end{equation}
\newline 
\indent We now want to write 
\begin{multline}
 \frac{e^{\pi\sqrt{\frac{2\alpha}{3}(N+1)}}}{(N+1)^{\frac{\alpha}{4}+\frac 34}} 
\left( 1+\frac{c_{\alpha,1}}{(N+1)^\frac 12} + \frac{c_{\alpha, 2}}{N+1} + \frac{c_{\alpha,3}}{(N+1)^{\frac 32}} + \frac {D_{\alpha, 1}(N+1)}{(N+1)^{2}} \right)\\
=
\frac{ e^{\pi\sqrt{\frac{2\alpha}{3}N}} }{N^{\frac{\alpha}{4}+\frac 34}}
\left( 1+\frac{A_{\alpha,1}}{N^\frac 12} + \frac {A_{\alpha,2}}{N} + \frac{A_{\alpha,3}}{N^{\frac 32}}+ \frac{B_\alpha(N)}{N^{2}} \right), \label{Ba expansion}
\end{multline}
where we need $A_{\alpha,1}$, $A_{\alpha, 2}$, and $A_{\alpha,3}$ explicitly and a bound on $B_\alpha(N)$. To find the $A_{\alpha, j}$'s,  we use \eqref{NA} to rewrite the powers of $N+1$ on the left-hand side in terms of powers of $N$ and employ \eqref{Nex} to rewrite the exponential term. In doing so and comparing powers of $N$ on each side, one concludes that 
\begin{align}\label{Aadef}
&A_{\alpha,1}=c_{\alpha,1}+\frac{\pi{\alpha^\frac12}}{\sqrt 6},\quad A_{\alpha,2}= c_{\alpha,2}+\frac{\pi{\alpha^\frac12}}{\sqrt{6}}c_{\alpha,1}+ \frac{\pi^2\alpha}{12} -\frac \alpha4 - \frac 34,\\
&A_{\alpha,3} = c_{\alpha,3} + \frac{\pi {\alpha^\frac12}}{\sqrt{6}} c_{\alpha,2} +   \left( \frac{\pi^2\alpha}{12} - \frac \alpha4 - \frac 54\right)c_{\alpha,1} + \frac{\pi^3 \alpha^{\frac 32} - 9\pi {\alpha^\frac12}}{36\sqrt{6}} - \frac{\pi {\alpha^\frac12}}{\sqrt{6}}\left(\frac \alpha4 + \frac 34\right). \nonumber
\end{align}
Below, we need bounds on each of these quantities. By the triangle inequality and the bounds in \eqref{ca}, one can find that 
\begin{align}\label{Aabounds}
|A_{\alpha,1}| \le \frac{17\alpha^{\frac 32}}{20}, \quad
|A_{\alpha,2}| \le \frac{3\alpha^3}{8}, \quad
|A_{\alpha,3}| < \frac{9\alpha^{\frac 92}}{40}.
\end{align}
We next bound the error term $B_\alpha(N)$. We can solve for $\frac{B_\alpha(N)}{N^2}$ in \eqref{Ba expansion} as 
\begin{multline*}
 \frac{e^{\pi\sqrt{\frac{2\alpha}{3}(N+1)}} }{(N+1)^{\frac{\alpha}{4}+\frac 34}}
\left( 1+\frac{c_{\alpha,1}}{(N+1)^\frac 12} + \frac{c_{\alpha,2}}{N+1} + \frac{c_{\alpha,3}}{(N+1)^{\frac 32}} + \frac {D_{\alpha,1}(N+1)}{(N+1)^{2}} \right) 
N^{\frac{\alpha}{4}+\frac 34}e^{-\pi\sqrt{\frac{2\alpha}{3}N}} 
\\-1-\frac{A_{\alpha,1}}{{N^\frac 12}} - \frac{A_{\alpha,2}}{N} - \frac{A_{\alpha,3}}{N^{\frac 32}}
\\
=
e^{\pi\sqrt{\frac{2\alpha}{3}(N+1)}}\Bigg(\frac{1}{(N+1)^{\frac{\alpha}{4}+\frac 34}}+\frac{c_{\alpha,1}}{(N+1)^{\frac{\alpha}{4}+\frac 54}}+\frac{c_{\alpha,2}}{(N+1)^{\frac \alpha4 + \frac 74}} + \frac{c_{\alpha,3}}{(N+1)^{ \frac \alpha4 + \frac 94}}
+\frac{D_{\alpha,1}(N+1)}{(N+1)^{\frac{\alpha}{4}+\frac{11}{4}}}\Bigg)
\\\times
N^{\frac{\alpha}{4}+\frac 34}e^{-\pi\sqrt{\frac{2\alpha}{3}N}}
-1-\frac{A_{\alpha,1}}{{N^\frac 12}} - \frac{A_{\alpha,2}}{N} - \frac{A_{\alpha,3}}{N^{\frac 32}}.
\end{multline*}
Similar to finding the $A_{\alpha}$ above, we use \eqref{Nex} and \eqref{NA} to expand this as 
\begin{multline*}
\left(1+\frac{\pi{\alpha^\frac 12}}{\sqrt{6}N^{\frac 12}}+\frac{\pi^2\alpha}{12N}+ \frac{\pi^3 \alpha^{\frac 32} - 9\pi{\alpha^\frac 12}}{36\sqrt{6}N^{\frac 32}}+\frac{D_{\alpha,3}(N)}{N^{2}}\right)\\
\times\Bigg(1-\frac{\frac \alpha4 + \frac 34}{N} + \frac{D_{\frac \alpha4 + \frac 34,2}}{N^2}+\frac{c_{\alpha,1}}{N^{\frac{1}{2}}} - \frac{c_{\alpha,1}\left( \frac \alpha4 + \frac 54  \right) }{N^{\frac{3}{2}}} + \frac{c_{\alpha,1}D_{\frac \alpha4 + \frac 54,2}}{N^{\frac 52}} + \frac{c_{\alpha,2}}{N} 
\\
 + \frac{c_{\alpha,2}D_{\frac \alpha2 + \frac 74,2}^*}{N^2}  + \frac{c_{\alpha,3}}{N^{\frac 32}} + \frac{c_{\alpha,3}D_{\frac \alpha4 + \frac 94,2}^*}{N^{\frac 52}}+\frac{D_{\alpha,1}(N+1)D_{\frac{\alpha}{4}+\frac{11}{4},2}^{**}}{N^{2}}\Bigg)
-1-\frac{A_{\alpha,1}}{{N^\frac 12}} - \frac{A_{\alpha,2}}{N} - \frac{A_{\alpha,3}}{N^{\frac 32}}.
\end{multline*}
From here, we can simply expand out this product. Because all terms with power greater than $N^{-2}$ are already subtracted, we can factor this out of everything remaining and bound the absolute value of what is left using \eqref{Da2}, \eqref{ca}, \eqref{Da3}, and the fact that $N \ge 2\alpha^{11}$ to obtain a bound on $B_\alpha(N)$, namely
\begin{equation}\label{Ba}
|B_\alpha(N)| \le \frac{7\alpha^6}{27}.
\end{equation}
Now, using \eqref{I expansion} and \eqref{Ba expansion}, \eqref{main} becomes
\begin{align} \label{main term}
&\left(\frac{\alpha}{24}\right)^{\frac{\alpha}{2}+1} \sqrt{\frac{6}{\alpha}} \frac{e^{\pi\sqrt{\frac{2\alpha}{3}}\left(\sqrt{N}+\sqrt{L}\right)}}{N^{\frac{\alpha}{4}+\frac 34}L^{\frac{\alpha}{4}+\frac 34}}\\\nonumber
&\times
\Bigg(\left(1+\frac{c_{\alpha,1}}{{N^\frac12}} + \frac{c_{\alpha,2}}{N} + \frac{c_{\alpha,3}}{N^{\frac 32}}+\frac{D_{\alpha,1}(N)}{N^{2}}\right)\left(1+\frac{A_{\alpha,1}}{{L^\frac12}}+ \frac{A_{\alpha,2}}{L} + \frac{A_{\alpha,3}}{L^{\frac 32}}+\frac{B_\alpha(L)}{L^{2}}\right)\\&\nonumber
-\left(1+\frac{c_{\alpha,1}}{{L^\frac12}}+\frac{c_{\alpha,2}}{L} + \frac{c_{\alpha,3}}{L^{\frac 32}}+\frac{D_{\alpha,1}(L)}{L^{2}}\right)\left(1+\frac{A_{\alpha,1}}{{N^\frac12}} + \frac{A_{\alpha,2}}{N} + \frac{A_{\alpha, 3}}{N^{\frac 32}}+\frac{B_\alpha(N)}{N^{2}}\right)
\Bigg).
\end{align}
We write the expression in the outer parentheses as
\begin{align}\label{first term}
(A_{\alpha, 1} - c_{\alpha,1})\left(\frac{1}{{L^\frac12}}-\frac{1}{{N^\frac12}}\right) + f(N,L) = \frac{\pi{\alpha^\frac12}\left( {N^\frac12} - {L^\frac12} \right)}{\sqrt{6}N^\frac12L^\frac12}\left(1 + \frac{N^\frac12 L^\frac12 f(N,L)\sqrt{6}}{\pi{\alpha^\frac12}\left( {N^\frac12} - {L^\frac12} \right)} \right)
\end{align}
for some function $f(N,L)$, where the equality follows from \eqref{Aadef}. We wish to bound $\frac{f(N,L){N^\frac12L^\frac12}}{{\alpha^\frac12}( {N^\frac12} - {L^\frac12} )}$. Note that $f(N,L)$ can be easily calculated from \eqref{main term} by simply expanding the products. We do not write out every term but instead explain how to bound just a couple of the terms. For example, the next largest term (asymptotically) that arises when computing $f(N,L)$ is 
\begin{equation*}
(A_{\alpha,2}-c_{\alpha,2})\left(\frac 1L - \frac 1N\right).
\end{equation*}
We can bound the first product above using \eqref{Aadef} and \eqref{ca} as 
\begin{equation*}
|A_{\alpha, 2} - c_{\alpha, 2}| \le \frac \alpha4 + \frac 34 + \frac{\pi^2 \alpha}{12} + \frac{\pi{\alpha^\frac12}}{\sqrt{6}} |c_{\alpha, 1}|\le \alpha^2,
\end{equation*}
where the inequality follows from basic calculus. Then when $f(N,L)$ is multiplied by $\frac{{N^\frac12L^\frac12}}{{\alpha^\frac12}( {N^\frac12} - {L^\frac12} )}$, this term can be bounded in absolute value using \eqref{Aabounds} and \eqref{ca} by (using that $N > L \ge 2\alpha^{11}$) 
\begin{equation*}%\label{term1}
|A_{\alpha,2}- c_{\alpha,2}| \frac{{N^\frac12} + {L^\frac12}}{{ \alpha^\frac12 N^\frac12L^\frac12}} \le  \frac{2\alpha^2{N^\frac12}}{{\alpha^\frac12 N^\frac12L^\frac12}} \le \frac{\sqrt{2}}{\alpha^{4}}.
\end{equation*}
All of the exact terms that arise in $f(N,L)$ (i.e., those not involving the error terms $D_{\alpha,1}(N)$, $D_{\alpha,1}(L)$, $B_\alpha(N)$, or $B_\alpha(L)$) can be bounded in this way. As for the remaining terms of $f(N,L)$, one can use that $\frac{{N^\frac12L^\frac12}}{{N^\frac12} - {L^\frac12}}$ is decreasing as a function of $N$ to bound it above by 
\noindent
\begin{equation}\label{multbound}
\frac{{N^\frac12L^\frac12}}{{N^\frac12} - {L^\frac12}} \le \frac{{(L+1)^\frac12L^\frac12}}{{(L+1)^\frac12}-{L^\frac12}} \le \frac{201 L^{\frac 32}}{100},
\end{equation}
where the second inequality holds for $L \ge 2\alpha^{11} \ge 2^{12}$ by calculus. This allows us to bound all of the other terms of $f(N,L)$. For example, one of the remaining terms is
\begin{equation*}
\left(B_{\alpha}(L) - D_{\alpha,1}(L)\right) \frac{1}{L^2}.
\end{equation*}
Hence, when $f(N,L)$ is multiplied by $\frac{{N^\frac12L^\frac12}}{{\alpha^\frac12}( {N^\frac12} - {L^\frac12} )}$, this term can be bounded utilizing \eqref{multbound}, \eqref{Ba}, and \eqref{ca} by (using that $L \ge 2\alpha^{11}$)
\begin{equation*}%\label{bterm1}
|B_\alpha(L) - D_{\alpha,1}(L)| \frac{201}{100{\alpha ^\frac12L^\frac12}} \le \left( \frac{7\alpha^6}{27} + \frac{3\alpha^6}{25} \right) \frac{201}{100\sqrt{2}\alpha^{6}} = \frac{4288}{5625\sqrt{2}}.
\end{equation*}
All of the other terms are bounded in a similar manner (using the fact that $N \ge L$). Combining these bounds and using that $\alpha \ge 2$, we obtain
\begin{equation}\label{f bound}
\left| \frac{\sqrt{6}N^\frac12L^\frac12}{\pi{\alpha^\frac12}\left( {N^\frac12} -{L^\frac12} \right)} f(N,L) \right| \le \frac{13}{14}.
\end{equation}
Combining \eqref{main term}, \eqref{first term}, and \eqref{f bound}, the main term can be written as
\begin{equation} \label{final main}
\pi \left(\frac{\alpha}{24}\right)^{\frac{\alpha}{2}+1} \left({N^\frac12} - {L^\frac12}\right) \frac{e^{\pi\sqrt{\frac{2\alpha}{3}}\left(\sqrt{N}+\sqrt{L}\right)}}{N^{\frac{\alpha}{4}+\frac 54}L^{\frac{\alpha}{4}+\frac 54} }
\left(1 + O_\le\left(\frac{13}{14}\right)\right).
\end{equation}
We now need to bound the remaining terms in the expansion of $p_\alpha(n-1)p_\alpha(\ell+1) - p_\alpha(n)p_\alpha(\ell)$ coming from Theorem \ref{theorem}. To estimate the contribution from $k\geq 2$, we bound, for $X\in\R^+$
\begin{equation*}
F_{\kappa}(X):=\sum_{k\geq 2} I_\kappa\left(\frac{X}{k}\right).
\end{equation*}
Note that $F_\kappa(X)$ is monotonically increasing because $I_\kappa$ is. We estimate the first $\lfloor X\rfloor-1$ terms using Lemma \ref{Bessel} (1)
\begin{equation}\label{sum1}
\sum_{2\leq k \leq \lfloor X\rfloor} I_\kappa\left(\frac{ X }{k}\right)
\leq \sqrt{\frac{2}{\pi \lfloor X\rfloor}}\sum_{2\leq k \leq \lfloor X\rfloor} \sqrt{k}e^{\frac{ X }{k}}
\leq 2\sqrt{\frac{X}{\pi}}e^{\frac{ X }{2}}.
\end{equation}
For the second bound, we are using that
$\sqrt{k}e^{\frac{ X }{k}}\leq \sqrt{2}e^{\frac{ X }{2}}$.
To bound the remaining terms of $F_\kappa(X)$, we use Lemma \ref{Bessel} (3) to conclude that
\begin{align} \label{sum2} 
\sum_{k\geq \lfloor X\rfloor+1} I_\kappa \left(\frac{ X }{k}\right)
\leq \frac{2^{1-\kappa}X^\kappa}{\Gamma(\kappa + 1)} \sum_{k\geq \lfloor X\rfloor+1} \frac{1}{k^{\kappa}}
\leq \frac{2^{1-\kappa}X^\kappa }{\Gamma(\kappa + 1)} \int_{\lfloor X\rfloor}^{\infty}\frac{1}{x^{\kappa}}dx \leq \frac{2^{1-\kappa} X}{(\kappa-1)\Gamma(\kappa+1)}.
\end{align}
Combining \eqref{sum1} and \eqref{sum2} and using basic calculus and the fact that $\kappa \ge 2$, we determine that 
\begin{equation}\label{F bound}
F_{\kappa}(X) \le 4 \sqrt{\frac{X}{\pi}}e^{\frac{X}{2}}.
\end{equation}
Using this, we may bound the non-main terms as follows. We first bound the non-main terms corresponding to $p_\alpha (n-1)p_\alpha(\ell+1)$. For this, we consider terms with (i) $k_1, k_2 \geq 2$, (ii) terms with $k_1 \geq 2$ and $k_2 = 1$, (iii) terms with $k_1 = 1$ and $k_2 \geq 2$, (iv) terms with $k_1 = k_2 = 1$ and $m_1 \geq 1$, and finally (v) terms with $k_1 = k_2 = 1$, $m_1 = 0$, and $m_2 \geq 1$. As is done above, we do not write out all of these sums but instead just illustrate how to bound the terms corresponding to (ii). In what follows, we let $\beta := \lfloor \frac{\alpha}{24}\rfloor$ to simplify notation, and we get an upper bound of 
\begin{align}\label{termii}
&\frac{4\pi^2 }{N^{\frac{\alpha}{4}+\frac 12}(L+1)^{\frac{\alpha}{4}+\frac 12}}
\sum_{0 \le m_1, m_2 \le \beta} \left(\frac{\alpha}{24}-m_1\right)^{\frac{\alpha}{4}+\frac 12} \left(\frac{\alpha}{24}-m_2\right)^{\frac{\alpha}{4}+\frac 12} p_\alpha(m_1)p_\alpha(m_2) \\
&\qquad\qquad\qquad\qquad\times
F_{\frac{\alpha}{2}+1}\left(4\pi\sqrt{\left(\frac{\alpha}{24}-m_1\right)N}\right)
I_{\frac{\alpha}{2}+1}\left(4\pi\sqrt{\left(\frac{\alpha}{24}-m_2\right)(L+1)}\right) \notag \\
&\le \frac{4\pi^2 }{N^{\frac{\alpha}{4}+\frac 12}(L+1)^{\frac{\alpha}{4}+\frac 12}}
(\beta + 1)^2 \left(\frac{\alpha}{24}\right)^{\frac{\alpha}{2}+1}  p_\alpha(\beta)^2
F_{\frac{\alpha}{2}+1} 
\left(\pi\sqrt{\frac{2\alpha}{3}N}\right)
I_{\frac{\alpha}{2}+1}\left(\pi\sqrt{\frac{2\alpha}{3}(L+1)}\right), \notag
\end{align}
where the inequality follows from the monotonicity of $F$ and $I$. Using \eqref{F bound} and Lemma \ref{Bessel} (1), we can further bound \eqref{termii} by
\begin{equation}\label{termii2}
16\sqrt{2}\pi   (\beta + 1)^2 \left( \frac{\alpha}{24} \right)^{\frac \alpha2 + 1} p_\alpha(\beta)^2 
\frac{e^{\pi \sqrt{\frac{\alpha}{6}N} + \pi\sqrt{\frac{2\alpha}{3}(L+1)}}}{N^{\frac \alpha4 +\frac 14} (L+1)^{ \frac \alpha4 + \frac 34}}.
\end{equation}
Now, we claim that 
\begin{equation*}
	p_\alpha(n) \le e^{\pi \sqrt{\frac{2\alpha n}{3}}}.
\end{equation*}
This follows from virtually the same proof as in the $\alpha = 1$ case; see \cite{A} for details. Using this bound on $p_\alpha(n)$, the fact that $\sqrt{L} \le \sqrt{L+1} \le \sqrt{L} + 1$, and factoring out the terms outside of parentheses in \eqref{final main}, we see that \eqref{termii2} is at most 
\begin{equation*}\label{termii3}
\pi \left( \frac{\alpha}{24} \right)^{\frac \alpha2 + 1} \left( {N^\frac12} - {L^\frac12} \right) \frac{e^{\pi\sqrt{\frac{2\alpha}{3}N} + \pi\sqrt{\frac{2\alpha}{3}L}}}{N^{\frac \alpha4 +\frac 54} L^{ \frac \alpha4 +\frac 54}}    \frac{16\sqrt{2} (\beta + 1)^2 NL^{\frac 12}}{ {N^\frac 12} - {L^\frac12}}
e^{ \frac{\pi \alpha}{3} + \pi\sqrt{\frac{2\alpha}{3}} - \pi\sqrt{\frac{\alpha}{6}N}}.
\end{equation*}
Now, we are left to bound
\begin{equation*}\label{boundbyconstant}
\frac{16\sqrt{2} (\beta + 1)^2 NL^{\frac 12}}{ {N^\frac 12} - {L^\frac12}}
e^{ \frac{\pi \alpha}{3} + \pi\sqrt{\frac{2\alpha}{3}} - \pi\sqrt{\frac{\alpha}{6}N}}
\end{equation*}
by a constant. To do so, we again use 
\eqref{multbound} to obtain an upper bound of
\begin{equation}\label{interterm}
\frac{804}{25}\sqrt{2}(\beta +1)^2 {N^\frac12} L^{\frac 32}
e^{ \frac{\pi \alpha}{3} + \pi\sqrt{\frac{2\alpha}{3}} - \pi\sqrt{\frac{\alpha}{6}N}}.
\end{equation}
It is easy to check that this is decreasing in $N$ for $N \ge 2\alpha^{11} \ge 2^{12}$, so utilizing that $N \ge L$, we get an upper bound on \eqref{interterm} of
\begin{equation*}\label{termiiLdec}
\frac{804}{25} \sqrt{2} (\beta + 1)^2 L^2
e^{ \frac{\pi \alpha}{3} + \pi \sqrt{\frac{2\alpha}{3}} - \pi \sqrt{\frac{\alpha}{6}L}}.
\end{equation*}
Similarly, this term is decreasing in $L$ for $L \ge 2\alpha^{11} \ge 2^{12}$, so we can plug in $L = 2\alpha^{11}$ to obtain a bound of 
\begin{equation*}\label{termiiadec}
\frac{3216}{25} \sqrt{2} (\beta + 1)^2 \alpha^{22}
e^{ \frac{\pi \alpha}{3} + \pi\sqrt{\frac{2\alpha}{3}} - \frac{\pi \alpha^6}{\sqrt{3}}}.
\end{equation*}
One can bound the exponent by $-\frac{17}{10} \alpha^6$, and estimate $\beta + 1 \le \frac{13\alpha}{24}$. The resulting term is
\begin{equation*}
\frac{11323}{300} \sqrt{2} \alpha^{24}e^{-\frac{17}{10} \alpha^6}.
\end{equation*} 
Using that this expression is decreasing in $\alpha$, we get an upper bound by plugging in $\alpha = 2$ yielding a numerical answer of $51\cdot 10^{-40}$. One can similarly bound all of the other error terms; the only significant departure occurs when bounding terms corresponding to (iv) and (v), where a term $(\frac{\alpha}{24} - 1 )^{-\frac 14}$ occurs. It is here that we need to use the bound $L \ge \frac{100}{\alpha - 24}$ to ensure that our argument of $L$ is large enough. The proof in this case is still similar in nature and is omitted. The largest of the errors that arise from these cases is $10^{-4}$, which when combined with the error of \eqref{final main} gives the statement of the theorem.
\end{proof}
\section{Proofs of the corollaries}\label{corollary}
As alluded to after the statement of Theorem~\ref{main theorem},  all of the terms in the expansion of $p_\alpha(n-1)p_\alpha(\ell+1) - p_\alpha(n)p_\alpha(\ell)$ are positive, so Corollary~\ref{eventual} follows.  Thus, only the proof of Corollary~\ref{log-concavity} remains.  In order to prove this for a fixed value of $k$, we only need to compute the ratios $\frac{p_k(\ell + 1)}{p_k(\ell)}$ up to $\ell \ge \lceil 2k^{11} + \frac{k}{24} \rceil$ and see that they are decreasing, except for $k=2$. For $k\in\{2,3\}$, we can do this directly, but for $k \in \{4,5\}$, we need to find a way to make the computation more efficient and store less memory; we provide the necessary details to do so in the following subsection. At the end of the section, we describe how proving the result for $k\in\{2,3,4, 5\}$ is sufficient to prove Corollary~\ref{log-concavity}. Namely, in Proposition~\ref{endprop}, we show that $p_k(n)$ is log-concave for $k \in \{3,4,5\}$ and point out that convolution of log-concave sequences is log-concave, which shows that the same property also holds for $k \in \N_{\ge 3}$. 

\subsection{Tools needed for the proof of Corollary~\ref{log-concavity}} 
  To verify the initial cases of the conjecture, of course a direct approach using Rademacher sums, recursive formulas, or by convoluting the partition generating function can be used. However, due to the large number of cases that have to be checked (for example, for $p_4(n)$ we need to compute all values with $n \le 2\cdot 4^{11} + 6$ ), these direct methods are not sufficient. An approach with lower time and memory requirements is thus essential in practice. As a result, we begin by defining sequences that approximate our partition numbers well enough to prove the lemma and which also require less memory and speed to compute. To do this, let $\bm{d}=\left(d_j\right)_{j=1}^{\infty}$ be a sequence of positive integers $d_j\leq j$, and for $n\in\N$ recursively define 
\begin{align*}
p_{k,\bm{d}}^{\pm}(0)&:=1 \\
p_{k,\bm{d}}^-(n)&:=\frac{k}{n}\sum_{\ell=1}^{d_n} \sigma(\ell) p_{k,\bm{d}}^{-}(n-\ell)&& \text{ for }n\geq 1,\\
p_{k,\bm{d}}^{+}(n)&:=\frac{k}{n}\sum_{\ell=1}^{d_n} \sigma(\ell) p_{k,\bm{d}}^{+}(n-\ell)+knp_{k,\bm{d}}^+\left(n-d_n-1\right) &&\text{ for }n\geq 1.
\end{align*}
We also set the negative values to be zero:
\[p_{k,\bm{d}}^{\pm}(n)=0 \text{ for }n\leq -1\]

\begin{lemma}\label{lem:pklowerupper}
	For $n\in\N$, we have 
	\[
	p_{k,\bm{d}}^-(n)\leq p_k(n)\leq p_{k,\bm{d}}^+(n).
	\]
\end{lemma}
\begin{proof}
Using (3) of \cite{HNA}, we find that for $n\geq 1$, we have
	\begin{equation}\label{eqn:pkeval}
	p_k(n)=\frac{k}{n}\sum_{\ell=1}^{n} \sigma(\ell)p_k(n-\ell).
	\end{equation}
	 We prove the claimed inequalities by induction. The base case, $n=0$, is trivial as $p_k(0)=1$. Assume inductively that for every $0\leq m<n$ the claim holds. Note that both $p_k(n)$ and $\sigma(n)$ are non-negative for all $n\in\N$. Hence the inequality $d_n\leq n$ and the inductive hypothesis $p_k(n-\ell) \ge p_{k,\bm{d}}^-(n-\ell)$ imply that  
	\[
	p_k(n)=\frac{k}{n}\sum_{\ell=1}^{n} \sigma(\ell) p_k(n-\ell)\geq \frac{k}{n}\sum_{\ell=1}^{d_n} \sigma(\ell) p_{k,\bm{d}}^{-}(n-\ell)=p_{k,\bm{d}}^-(n).
	\]
	This gives the first inequality.
	
	To obtain the upper bound, we note that $p_k(n)$ is increasing, and hence 
	\begin{align*}
	p_k(n)&= \frac{k}{n}\sum_{\ell=1}^{d_n} \sigma(\ell) p_k(n-\ell) + \frac{k}{n}\sum_{\ell=d_n+1}^{n} \sigma(\ell) p_k(n-\ell)\\
	&\leq  \frac{k}{n}\sum_{\ell=1}^{d_n} \sigma(\ell) p_k(n-\ell) + \frac{k}{n}p_k\left(n-d_n-1\right)\sum_{\ell=1}^{n} \sigma(\ell).
	\end{align*}
	We next bound 
	\[
	\sum_{\ell=1}^n\sigma(\ell) = \sum_{\ell=1}^n\sum_{d\mid \ell} d  = \sum_{d=1}^n d \sum_{1\leq j\leq \left\lfloor \frac{n}{d}\right\rfloor} 1 \leq \sum_{d=1}^n d \frac{n}{d} =n^2. 
	\]
	Therefore 
	\[
	p_k(n)\leq \frac{k}{n}\sum_{\ell=1}^{d_n} \sigma(\ell) p_k(n-\ell) + \frac{k}{n}p_k\left(n-d_n-1\right) n^2.
	\]
	Using the inductive hypothesis for the upper bounds, we have 
	\[
	p_k(n)\leq \frac{k}{n}\sum_{\ell=1}^{d_n} \sigma(\ell) p_{k,\bm{d}}^+(n-\ell) + knp_{k,\bm{d}}^+\left(n-d_n-1\right)=p_{k,\bm{d}}^+(n).\qedhere
	\]
\end{proof}
\begin{remark*}
In the special case $d_n=n$, one has $p_{k,\bm{d}}^-(n)=p_{k}(n)=p_{k,\bm{d}}^+(n)$ by \eqref{eqn:pkeval}. In order to compute $p_{k,\bm{d}}^{\pm}(n)$ for every $1\leq n\leq N$, the number of steps required is
$
	O(\sum_{1\leq n\leq N} d_n).
$
Thus the number of steps to compute $p_k(n)$ directly (i.e., $d_n=n$) is $O(n^2)$. If $d_n\ll n^{\delta}$, then the number of steps to compute the lower and upper bounds $p_{k,\bm{d}}^{\pm}(n)$ is $\ll N^{1+\delta}$. Moreover, in order to compute $p_{k,\bm{d}}^{+}(n)$ with a computer one only needs to keep $d_n=O(n^{\delta})$ numbers in memory (this is $O(n)$ in the special case $d_n=n$). Hence computing the sequences $p_{k,\bm{d}}^{\pm}(n)$ is better than $p_k(n)$ both in the speed of the calculation and in the memory requirement.
\end{remark*}

These numbers grow very quickly. Thus, if $\bm{d}$ is chosen appropriately so that $n$ is small in comparison with an exponential of the shape
\[
e^{2\pi c_k\left(\sqrt{n}-\sqrt{n-d_n-1}\right)},\quad (c_k>0)
\]
then we expect a good approximation of $p_k(n)$. Indeed, from Theorem~\ref{theorem}, we have
\[
p_{4}(n) \sim \frac{e^{2\pi \sqrt{\frac{2n}{3}}}}{2^{\frac{7}{4}}3^{\frac{5}{4}}n^{\frac{7}{4}}}\left(1-\left(\frac{35\sqrt{3}}{16\pi} + \frac{\pi}{3\sqrt{3}}\right)n^{-\frac{1}{2}}\right).
\]
Hence in this case we need to compare $n$ against 
\[
e^{2\pi \left(\sqrt{\frac{2n}{3}}- \sqrt{\frac{2\left(n-d_n-1\right)}{3}}\right)}.
\]
Using Taylor's Theorem, we see that for $d_n < n-1$, 
\begin{equation*}
\sqrt{\frac{2\left(n-d_n-1\right)}{3}}= \sqrt{\frac{2n}{3}}-\frac{1}{\sqrt{6}}\frac{d_n+1}{\sqrt{n}}+O\left(\frac{d_n^2}{n^{\frac{3}{2}}}\right).
\end{equation*}
Then for $d_n\sim  n^{\frac{1}{2}+\delta}$, this becomes
\[
\sqrt{\frac{2n}{3}}-\frac{1}{\sqrt{6}}\frac{d_n+1}{\sqrt{n}}+O\left(n^{2\delta-\frac{1}{2}}\right).
\]
For $0<\delta<\frac{1}{4}$, we have
\[
e^{2\pi \left(\sqrt{\frac{2n}{3}}- \sqrt{\frac{2\left(n-d_n-1\right)}{3}}\right)}= e^{-\frac{2\pi }{\sqrt{6}}\frac{d_n+1}{\sqrt{n}}+O(1)}\ll e^{-\frac{2\pi}{\sqrt{6}} n^{\delta}}.
\]
Hence by choosing $\delta$ appropriately, we can get any fixed number of digits of accuracy that are needed for a calculation.

\rm

\begin{lemma}\label{lem:pkpmlogconcave}
	The sequence $p_{k}(n)$ is log-concave if and only if there exists a sequence $\bm{d}$ of positive integers with $d_j\leq j$ such that for every $n\in\N$
	\begin{equation*}%\label{eqn:p+-rat}
	\frac{p_{k,\bm{d}}^-(n)}{p_{k,\bm{d}}^+(n-1)}\geq \frac{p_{k,\bm{d}}^+(n+1)}{p_{k,\bm{d}}^-(n)}.
	\end{equation*}
\end{lemma}
\begin{proof}
	By Lemma \ref{lem:pklowerupper}, for every such $\bm{d}$ and every $m\in\N_0$ we have 
	\[
	p_{k,\bm{d}}^-(m)\leq p_{k}(m)\leq p_{k,\bm{d}}^+(m).
	\]
	We conclude that for every $\bm{d}$ and every $m\in\N$ we have 
	\begin{align}
	\label{eqn:ratlower} \frac{p_{k}(m)}{p_{k}(m-1)} &\geq \frac{p_{k,\bm{d}}^-(m)}{p_{k,\bm{d}}^+(m-1)},\\
	\label{eqn:ratupper} \frac{p_{k}(m)}{p_{k}(m-1)} &\leq \frac{p_{k,\bm{d}}^+(m)}{p_{k,\bm{d}}^-(m-1)}.
	\end{align}
	Hence if such a choice of $\bm{d}$ exists for which the lemma holds, then \eqref{eqn:ratlower} with $m=n$ and \eqref{eqn:ratupper} with $m=n+1$ imply that for every $n\in\N$
	\[
	\frac{p_{k}(n)}{p_{k}(n-1)}{\geq} \frac{p_{k,\bm{d}}^-(n)}{p_{k,\bm{d}}^+(n-1)}{\geq} \frac{p_{k,\bm{d}}^+(n+1)}{p_{k,\bm{d}}^-(n)}{\geq} \frac{p_{k}(n+1)}{p_{k}(n)},
	\]
	and we see that $p_k(n)$ is log-concave. The converse follows by taking $d_j = j$ since then $p_{k, \bm{d}}^-(n) = p_k(n) = p_{k,\bm{d}}^+ (n)$ by using the definitions of $p_{k,\bm{d}}^-(n), p_{k,\bm{d}}^+(n)$ and using \eqref{eqn:pkeval}.
\end{proof}
\begin{remark*}
	Since the sequences $p_{k,\bm{d}}^{\pm}(n)$ are generally faster to compute than $p_{k}(n)$ and have a smaller memory requirement, in practice Lemma \ref{lem:pkpmlogconcave} gives us an easier and faster criterion to check to numerically verify the log-concavity of $p_k(n)$ for $1\leq n\leq N$ for some fixed $N$. 
\end{remark*}

\begin{proposition}\label{endprop}
	\noindent
	
	\noindent
	\begin{enumerate}[wide, labelwidth=!, labelindent=0pt] 
		\item[\normalfont(1)]
		For every $\alpha\in\R^+$, if $p_{\alpha}(n)$ satisfies the inequality
		\[
		\frac{p_{\alpha}(\ell+1)}{p_{\alpha}(\ell)} \geq \frac{p_{\alpha}(n)}{p_{\alpha}(n-1)}
		\]
		for every $n>\ell\geq 0$, then for every $j_1\in\N_0$ and $j\in\N_0\setminus\{1,2\}$ the sequences $p_{j_1\alpha+j}(n)$  satisfy the same inequality. 
		
		\item[\normalfont(2)] In particular, Conjecture \ref{Conj CFT} is true.
		
	\end{enumerate}
\end{proposition}
\begin{proof}
	(1) As remarked in Section~\ref{intro}, the claimed inequality is equivalent to log-concavity, i.e.,
	\begin{equation*}%\label{eqn:logconcaveineq}
	\frac{p_{\alpha}(n)}{p_{\alpha}(n-1)} \geq \frac{p_{\alpha}(n+1)}{p_{\alpha}(n)}
	\end{equation*}
	for $n \ge 1$. By Corollary~\ref{eventual}, it is true for $n \ge  2\alpha^{11} + \frac{\alpha}{24} + 1$ if $\alpha \in \{3, 4, 5\}$. We explicitly compute $p_{3}(n)$ for $n\leq 2\cdot 3^{11}+1$ and verify the inequality for those cases. Thus $p_3(n)$ is log-concave.
	
	Using heuristics based on the asymptotic growth of the coefficients $p_4(n)$, we let $\bm{d}=\bm{d}_4$ be the sequence
	\[
	d_j=d_{4,j}:=\begin{cases} j&\text{if }j\leq 2\cdot 10^5,\\[5 pt]
	\left\lfloor 250j^{\frac{1}{3}}\right\rfloor&\text{if }2\cdot 10^5<j\leq 3.5\cdot 10^6,\\[10pt]
	\left\lfloor 1125 j^{\frac{1}{3}}\right\rfloor&\text{if }j> 3.5\cdot 10^6.
	\end{cases}
	\]
	Using a computer (a Lenovo Thinkstation P330 with Intel core i7-9700 processor and 32GB memory), a 5-day-long calculation verifies that 
	\[
	\frac{p_{4,\bm{d}}^-(n)}{p_{4,\bm{d}}^+(n-1)}\geq \frac{p_{4,\bm{d}}^+(n+1)}{p_{4,\bm{d}}^-(n)}
	\]
	holds for all $n\leq 8.5\cdot 10^6$. Hence, by Lemma \ref{lem:pkpmlogconcave}, we see that for $n\geq 0$,
	\[
	\frac{p_4(n)}{p_{4}(n-1)}\geq \frac{p_{4}(n+1)}{p_{4}(n)}.
	\]
Similarly, estimating the asymptotic growth of the coefficients $p_5(n)$, we let $\bm{d}=\bm{d}_5$ be the sequence
	\[
	d_j=d_{5,j}:=\begin{cases} j&\text{if }j\leq 8\cdot 10^5,\\[5 pt]
	\left\lfloor 25j^{\frac{1}{2}}\right\rfloor&\text{if }8\cdot 10^5<j\leq 2\cdot 10^7,\\[10pt]
	\left\lfloor \frac{43}{2} j^{\frac{1}{2}}\right\rfloor&\text{if }j> 2\cdot 10^7.
	\end{cases}
	\]
	Using the same computer as before, a 71-day-long calculation (producing over 100GB of output) verifies that 
	\[
	\frac{p_{5,\bm{d}}^-(n)}{p_{5,\bm{d}}^+(n-1)}\geq \frac{p_{5,\bm{d}}^+(n+1)}{p_{5,\bm{d}}^-(n)}
	\]
	holds for all $n\leq 10^8$. Hence, by Lemma \ref{lem:pkpmlogconcave}, we see that for $n\geq 0$,
	\[
	\frac{p_5(n)}{p_{5}(n-1)}\geq \frac{p_{5}(n+1)}{p_{5}(n)}.
	\]

	By \cite[Theorem 1.4]{JG} (which the authors attributed to Hoggar \cite{Hoggar}), if two sequences satisfy log-concavity, then their convolution also satisfies log-concavity. Note that the convolution of $p_{\alpha_1}(n)$ and $p_{\alpha_2}(n)$ is precisely $p_{\alpha_1+\alpha_2}(n)$. Hence if $p_{\alpha_1}(n)$ and $p_{\alpha_2}(n)$ are both log-concave, then so is $p_{\alpha_1+\alpha_2}(n)$. Since the above shows that $p_3(n)$, $p_4(n)$, and $p_5(n)$ are all log-concave, we conclude that $p_{j_1\alpha+3j_2+4j_3+5j_4}(n)$ is log-concave for every $j_1,j_2,j_3,j_4\in\N_0$. The integers of the form $j=3j_2+4j_3+5j_4$ with $j_2,j_3,j_4\in \N_0$ are precisely $j\in\N_0\setminus\{1,2\}$. 
	
	\noindent
	(2) We may take $\alpha=3$ in (1) and obtain that $p_{3j_1+j}(n)$ is log-concave for every $j_1\in\N_0$ and $j\in\N_0\setminus\{1,2\}$. Each positive integer at least $3$ may be written in the form $3j_1+j$ with $j_1\in\N_0$ and $j\in \N_0\setminus\{1,2\}$. So the conjecture is true for every positive integer at least $3$. For $k=2$, the inequality from the conjecture is true for $n> 2^{12}+1$ by Corollary~\ref{eventual}, and a quick computer check verifies the claim for $n\leq 2^{12}+1$. 
\end{proof}

\begin{remark*}
While the above proposition is stated for all $\alpha \in \R^+$, one can directly compute $p_\alpha(n)$ for $n \in\{0, 1, 2\}$ and see that the inequality for log-concavity holds for $n = 1$ if and only if $\alpha \ge 3$. Hence, we should only concern ourselves with $\alpha \ge 3$ for log-concavity.
\end{remark*}

\section{Concluding remarks}\label{conclusion} 
We finish our paper with a list of possible follow-up ideas based on our work.

\begin{enumerate}[leftmargin=20pt]
	\item[\rm (1)] Use explicit bounds and convolution of series to prove log-concavity of other infinite families of sequences. The fact that the proof of Conjecture~\ref{Conj CFT} can be reduced to a finite check (instead of simply a finite check for each value of $k$) is surprising and may be able to be used to prove similar results.
	\item[\rm (2)] Find other interesting inequalities satisfied by the $k$-colored partition functions. As mentioned in Section~\ref{intro}, there are a number of papers studying analogues of the Bessenrodt--Ono inequality in various settings. There are likely a number of other inequalities to consider for $k$-colored partitions.
	\item[\rm(3)] Prove Conjecture~\ref{Conj HN} for intervals of $\alpha$. Using Proposition~\ref{endprop}, Corollary~\ref{log-concavity} could be extended to infinitely many other values of $\alpha$ (for example, by doing a computer check if $\alpha = 3.5$). However, the methods in this paper only appear to allow us to prove the conjecture for discrete sets of $\alpha$ via computer calculations.
	\item[\rm(4)] Give a combinatorial proof of Conjecture~\ref{Conj CFT}.
	\item[\rm(5)] Find explicit bounds for when the higher order Tur\'an inequalities hold for $p_\alpha(n)$. Chen, Jia, and Wang \cite{CJW} conjectured that inequalities beyond log-concavity eventually hold for the partition function, which was proven in \cite{GORZ}. This paper also tells us that these inequalities eventually hold for $p_\alpha(n)$. However, one could make these bounds explicit similar to how we have done here, which may show when exactly the inequalities begin to hold (see for example \cite{GORTTW, IJ}). 
\end{enumerate}

\end{document}